\documentclass[twoside,a4paper,leqno,12pt]{amsproc}
\usepackage[top=30mm,right=30mm,bottom=30mm,left=30mm]{geometry}
\usepackage[pagebackref]{hyperref}
\hypersetup{citecolor=blue, linkcolor=blue, colorlinks=true}
\usepackage{amsmath, amssymb, amsthm, amsfonts, color, amsrefs, booktabs}
\usepackage{graphicx, float, array, rotating, verbatim,stmaryrd}
\usepackage{mathtools,datetime}%,gensymb}
\renewcommand{\le}{\leqslant}
\renewcommand{\ge}{\geqslant}
\renewcommand{\leq}{\leqslant}
\renewcommand{\geq}{\geqslant}
\newcommand{\lhdeq}{\trianglelefteqslant}    %normaleq          <=     <\lhd
\newcommand{\Aut}{\textup{Aut}}
\newcommand{\cs}{{\rm cs}}
\newcommand{\Gbar}{\overline{G}}
\newcommand{\Hbar}{\overline{H}}
\newcommand{\la}{\langle}
\newcommand{\ra}{\rangle}
\newcommand{\Sym}{\textup{S}}
\newcommand{\Ker}{\textup{ker}}
\newcommand{\Z}{\mathrm{Z}}

\usepackage[shortlabels]{enumitem}
\setlist[enumerate]{label=\rm{(\alph*)}}

\theoremstyle{definition}
\newtheorem{definition}{Definition}
\newtheorem{remark}[definition]{Remark}
\theoremstyle{plain}

\newtheorem{theorem}[definition]{Theorem}
\newtheorem{proposition}[definition]{Proposition}
\newtheorem{lemma}[definition]{Lemma}

\begin{document}

\author{Mariagrazia Bianchi}
\address{\phantom{|}\kern-1cm Dipartimento di Matematica, Universita degli Studi Via Saldini 50, 20133 Milano, Italy. \emph{E-mail address:} {\tt\texttt{mariagrazia.bianchi@unimi.it}}}
\author{S.\,P. Glasby}
\author{Cheryl E. Praeger}
\address{\phantom{|}\kern-1cm S.\,P. Glasby and Cheryl E. Praeger, Centre for the Mathematics of Symmetry and Computation, University of Western Australia, 35 Stirling Highway, Perth 6009, Australia. \emph{E-mail addresses:} {\tt\texttt{\{stephen.glasby, cheryl.praeger\}@uwa.edu.au}}}

\thanks{\phantom{|}\kern-1cm Acknowledgements: SG and CP gratefully acknowledge support from the Australian Research Council Discovery Project DP190100450.
  MB acknowledges support from G.N.S.A.G.A.~(Indam) and thanks the
  Centre for the Mathematics of Symmetry and Computation (CMSC) for its~hospitality. This work began in the CMSC Research Retreat of 2019.
  CP also thanks the Isaac Newton Institute for Mathematical Sciences for support and hospitality during the programme ``Groups, Representations and applications: New Perspectives''. This program was supported by
EPSRC grant number EP/R014604/1. We are most grateful to Rachel Camina for suggesting an improvement to Proposition~\ref{T:coprime} and to
Avinoam Mann for alerting us to a number-theoretic lemma of John Thompson.
We thank the referee for carefully reading our paper and
suggesting improvements.
  \newline
  2010 Math Subject Classification: 20E-45, 20D-60. 
}

\makeatletter        % begin hack so date appears with amsproc
\def\@adminfootnotes{%
  \let\@makefnmark\relax  \let\@thefnmark\relax
  \ifx\@empty\@date\else \@footnotetext{\@setdate}\fi%%   <-- added
  \ifx\@empty\@subjclass\else \@footnotetext{\@setsubjclass}\fi
  \ifx\@empty\@keywords\else \@footnotetext{\@setkeywords}\fi
  \ifx\@empty\thankses\else \@footnotetext{%
    \def\par{\let\par\@par}\@setthanks}%
  \fi}\makeatother   % end hack so date appears

\title[]{Conjugacy class sizes in arithmetic progression}

% JPAA Submitted: 14 Mar, Rejected: 15 Mar
% J Group Theory Submitted: 25 Mar, Accepted 24 Apr (email from Chris Parker)

\subjclass[2010]{}
\date{\currenttime\ \today}                   % Activate to display current time\date{\today}

\begin{abstract}
  Let $\cs(G)$ denote the set of conjugacy class sizes of a group $G$,
  and let $\cs^*(G)=\cs(G)\setminus\{1\}$ be the sizes of non-central classes.
  We prove three results. We classify all finite groups for which
  (1)~$\cs(G)=\{a, a+d, \dots ,a+rd\}$ is an arithmetic progression with
  $r\ge2$; (2)~$\cs^*(G)=\{2,4,6\}$ the smallest case where $\cs^*(G)$
  is an arithmetic progression of length more than 2 (our most substantial
  result);  (3)~the largest two members of $\cs^*(G)$ are coprime. (For~(3)
  it is not obvious but it is true that $\cs^*(G)$ has two elements,
  and so is an arithmetic progression.)
%  (1) We classify all finite groups 
%  whose largest two non-central conjugacy class sizes are coprime.
%  (Here it is not obvious but it is true that
%  $\cs^*(G)$ has two elements, and so is an arithmetic progression.)
%  (2) As a corollary of this result we classify the groups $G$
%  with~$\cs(G)=\{a, a+d, \dots ,a+rd\}$ an arithmetic progression with
%  $r\ge2$. (3) Our most substantial result classifies all $G$
%  with~$\cs^*(G)=\{2,4,6\}$.
%  Let $\cs(G)$ denote the set of conjugacy class sizes of a group $G$,
%  and let $\cs^*(G)=\cs(G)\setminus\{1\}$ be the sizes of non-central classes.
%  We prove three results. We classify all finite groups $G$
%  with~$\cs(G)=\{a, a+d, \dots ,a+rd\}$ an arithmetic progression with
%  $r\ge2$. (We show $\cs(G)=\{1,2,3\}$.) Our most substantial result
%  classifies all $G$
%  with~$\cs^*(G)=\{2,4,6\}$. Finally, we classify all groups $G$
%  whose largest two non-central conjugacy class sizes are coprime.
%  (Here it is not obvious but it is true that
%  $\cs^*(G)$ has two elements, and so is an arithmetic progression.)
\end{abstract}

\maketitle

\section{Introduction}\label{s:intro}

There is a well-known but mysterious bijection between the set of
irreducible characters of a finite group $G$ and the set of
conjugacy classes of $G$.
(For the symmetric groups $\Sym_n$ the bijection is understood
via the partitions of $n$.) It is surprising that the set $\textup{cd}(G)$ of
degrees of irreducible characters (over ${\mathbb{C}}$) of $G$ and the
set $\cs(G)$ of sizes of
conjugacy classes of $G$ both seem to impose strong constraints
on the structure of $G$. Surveys of these topics~\cites{L,H,Camina}
state theorems where related hypotheses
on $\textup{cd}(G)$ and $\cs(G)$ give
rise to similar structural constraints~on~$G$.

Huppert~\cite{bH} shows that if $G$ satisfies
$\textup{cd}(G)=\{1,2,\dots,k\}$, then $k\in\{1,2,3,4,6\}$
and he describes such groups for each $k$. An analogous result shows that
if $\cs(G)=\{1,2,\dots,k\}$, then $k\in\{1,2,3\}$, see~\cite{B92}*{Theorem~1}.
In general, it is hard to classify all groups $G$ with a specified value of
$\cs(G)$. We do this when $\cs(G)=\{a_0,a_1,\dots,a_r\}$ is an
arithmetic progression where $a_i=a_0+id$ for $i\ge0$, and $a_0,d\ge1$ (Prop.~\ref{T:coprime}).
The result relies on a classification of groups $G$ whose largest two
class sizes are coprime (Theorem~\ref{T:KL}). The latter strengthens
both~\cite{BGC}*{Theorem} and~\cite{DJ}*{Theorem~(B1)}.

Henceforth, all our groups will be finite.
%
%The following theorem precisely determines
%the structure of the groups in~\cite{BGC}*{Theorem}
%and~\cite{DJ}*{Theorem (B1)}, thus strengthening these results.

\begin{theorem}\label{T:KL}
  Suppose that $G$ is a group with no non-trivial abelian direct factors and
  the largest two non-central conjugacy class sizes of $G$ are $m$ and $n$
  where $m<n$. Then $\gcd(m,n)=1$ if and only if
  $\cs(G)=\{1,m,n\}$, $G=K\rtimes L$ where $K$ is abelian
  $\gcd(|K|,|L|)=1$, $\Z(G)< L$, $L/\Z(G)$ is cyclic,
  $G/\Z(G)$ is a Frobenius group with
  kernel $K\Z(G)/\Z(G)$ and $m=|L:\Z(G)|$, $n=|K|$ satisfy $n\equiv 1\pmod m$.
\end{theorem}
%Theorem~\ref{T:coprime} is reminiscent of a result of

%An arithmetic progression $a_0,a_1,\dots,a_r$ has a
%\emph{common difference} $d$ if $a_i=a_0+id$ for $i\ge0$. Henceforth $a_0$ and
%$d$ will be positive integers, and all groups will be finite.
To suppress certain details, it is useful to consider
the set $\textup{cd}^*(G)$
of non-linear character degrees of $G$ and the set
$\cs^*(G)$ of non-central conjugacy class sizes. The former `ignores'
the derived quotient $G/G'$ and the latter  `ignores' the center~$\Z(G)$.

 Note that if $C$ is abelian then
 $\cs(G\times C)=\cs(G)$ and hence $\cs^*(G\times C)=\cs^*(G)$.
 Our main Theorem~\ref{T246} classifies groups $G$ with $\cs^*(G)=\{2,4,6\}$. (This is the
smallest case when $\cs^*(G)=\{a_0,a_0+d,a_0+2d\}$ as  $\cs^*(G)=\{2,3,4\}$
is excluded by Proposition~\ref{T:coprime}.) Our proof of
Theorem~\ref{T246} is both delicate and lengthy.
The case $\cs^*(G)=\{2,4,6,8\}$ was solved in~\cite{BGP}.

\begin{theorem}\label{T246}
  Suppose that $G$ is a finite group with no abelian direct
  factors and $\cs^*(G)=\{2,4,6\}$. Then $G=AB$ where $B\lhdeq G$,
  $|A|=2^\alpha$, $|B|=3$, $|A'|=2$, and $\Z(A)<C_A(B)<A$. Conversely,
  if $G$ has these properties, then $\cs^*(G)=\{2,4,6\}$.
\end{theorem}

The number $n_\alpha$ of groups of order $2^\alpha3$ with
$\cs^*(G)=\{2,4,6\}$ and no non-trivial abelian direct factors increases
quite rapidly with~$\alpha$. For example,
$n_\alpha=4, 16, 46, 104$ when $\alpha=5, 6, 7, 8$; see Remark~\ref{R246}.

Conjugacy classes $x^G$ of prime power size are important.
A beautiful theorem of Kazarin~\cite{K} says if $|x^G|$ is a prime power, then
$\langle x^G\rangle$ is a solvable subgroup of $G$.

In Section~\ref{s:conjsizes} we first prove Theorem~\ref{T:KL}.
Next we prove that, if $\cs(G)$ is an arithmetic progression of length
at least $3$,  then $\cs(G)=\{1,m,n\}$ satisfying the conditions of
Theorem~\ref{T:KL}, and from this we deduce the detailed structure of $G$
(see Proposition~\ref{T:coprime}).
%, and next prove Proposition~\ref{T:coprime} by showing that
%$\cs(G)=\{a_0,a_0+d,a_0+2d\}$ implies $\gcd(a_0+d,a_0+2d)=1$ and invoking
%Theorem~\ref{T:KL} to complete the proof. 
Section~\ref{s:ap} explores how
number theory constrains possible arithmetical
progressions involving precisely two primes, see
Lemma~\ref{L:NT} and Remark~\ref{R6}. Sections~\ref{sec:ex}
and \ref{sec:T246} give the proof of Theorem~\ref{T246}
when $\cs^*(G)=\cs(G)\setminus\{1\}$ equals $\{2,4,6\}$.
Call $r=|\cs^*(G)|$ the \emph{conjugacy rank} of $G$.

We remark that $\cs^*(G)$ can \emph{contain} arbitrarily long
arithmetic progressions. Consider $G_k=C_2\wr C_k$. Since $k,2^{k-1}\in\cs(G_k)$,
we see that $\{1,2,\dots, n\}\subseteq\cs(\prod_{k=1}^n G_k)$.
Also, $\cs(G)$ can equal a \emph{geometric} progressions of arbitrary length
by~\cite{CosseyHawkes}*{Theorem}.
%For example, the dihedral group $G_k=D_{2^k}$ or order $2^{k+1}$ has
%$\cs^*(G_k)=\{2,2^{k-1}\}$ and an inductive argument shows that
%$\{2,2^2,\dots,2^{\binom{n+1}{2}}\}=\cs^*(\prod_{k=1}^n G_k)$.
Indeed, given an arbitrary set $S$ of $p$-powers, \cite{CosseyHawkes}*{Theorem}
shows that there is a
$p$-group $G$ of class~$2$ with $\cs(G)=S$.

\section{Conjugacy class sizes}\label{s:conjsizes}

The set $\cs^*(S_n)$ of non-trivial conjugacy class sizes for the symmetric
group~$S_n$ below
\begin{table}[!ht]
  \begin{tabular}{cccccc}
    \toprule
    $n$&$2$&$3$&$4$&$5$&$6$\\
    $\cs^*(S_n)$&$\{\,\}$ &$\{2,3\}$&$\{3,6,8\}$&$\{10,15,20,24,30\}$&$\{15, 40, 45, 90, 120, 144\}$\\\bottomrule
  \end{tabular}
\end{table}
suggests that common divisors of class sizes is important. Indeed, the common divisor graph~\cite{Camina} plays a central role.
Note that the class equation has the form $|G|=\sum_{k \in \cs(G)} m_k k$ where
$m_k$ is the number of classes of $G$ of size $k$.

In this section we study groups $G$ with $\cs(G)=\{1, 1+d,\dots,1+rd\}$. 
It seems remarkable to the authors that, building on~\cite{BGC}, we can classify such $G$ if $r\ge2$.
Before giving our proof, we review some definitions and record some
useful facts.

A group $G$ is a \emph{Frobenius group} if it has a proper subgroup $H$ with
the property that $H\cap H^g=1$ for all $g\in G\setminus H$. Using
character theory, it can be shown that $H$ determines a normal subgroup $K$
satisfying $K\setminus\{1\}=\bigcap_{g\in G}(G\setminus H^g)$.
Observe that $G=KH$. We call $K$ the \emph{Frobenius kernel} and~$H$
the \emph{Frobenius complement} as $H\cap K=1$. The structure of $H$
and $K$ is severely 
constrained~\cite{I}*{Chapter~6}. For example, $K$ is nilpotent and if $|H|$ is even,
then $K$ is abelian. Moreover, the Sylow subgroups of $H$ are cyclic or
generalized quaternion 2-groups~\cite{I}*{Corollary 6.17}. For $x,g\in G$, let $x^g=g^{-1}xg$, $x^G=\{x^g\mid g \in G\}$ and $[x,g]=x^{-1}x^g$.

One way to suppress the role of abelian direct factors is to
focus on the classes of $G\setminus\Z(G)$ and study when
$\cs^*(G):=\cs(G)\setminus\{1\}$ is an arithmetic progression. Results such as
Proposition~\ref{T:coprime} and Theorem~\ref{T:KL}
affirm this decision, and experimental evidence shows that there
is a much richer family of groups for which
$\cs^*(G)$ is an arithmetic
progression rather than $\cs(G)$. Recall that $r=|\cs^*(G)|$ is called
the \emph{conjugacy rank} of $G$.
Groups with $r\le2$ have been well studied.
It\^o proved if $r=1$, then $G$ is nilpotent~\cite{Ito1}, and if
$r=2$, then $G$ is solvable~\cite{Ito2}.  
In addition, he proved in~\cite{Ito3} that if $G$ is simple and $r=3$, then
$G\cong\textup{SL}_2(2^f)$ with $f\ge2$.

We consider conjugacy rank three groups with $\cs^*(G)=\{a_0,a_0+d,a_0+2d\}$. Since
$a_0\ge2$ and $d\ge2$, the smallest example has $\cs^*(G)=\{2,4,6\}$.
(The possibility $\cs(G)=\{1,2,3,4\}$ with $d=1$ does not arise
by Proposition~\ref{T:coprime}.) There are many groups $G$ with
$\cs^*(G)=\{2,4,6\}$ (see Remark~\ref{R246}) and determining their
common features, to show that our necessary conditions are sufficient, was
a challenge.

\begin{proof}[Proof of Theorem~\ref{T:KL}]
  Suppose that $G$ is as in the hypothesis of Theorem~\ref{T:KL}
  (see Section~\ref{s:intro}), and recall the meaning of $m$ and $n$.
  It follows from~\cite{BGC}*{Corollary~2} that $G$ has conjugacy rank~$r=2$,
  so $\cs^*(G)=\{m,n\}$.
  Dolfi and Jabara~\cite{DJ}*{Theorem~A} characterise groups with $r=2$,
  into one of four types called (A), (B1), (B2), and (B3), and the connection
  with the class sizes is given in~\cite{DJ}*{Lemma~3.3}.
  Case (A) does not arise since $\gcd(m,n)=1$. Thus $G=KL$ where $K\lhdeq G$ and
  $\gcd(|K|,|L|)=1$. In case (B2), $L$ is a nonabelian $p$-group,
  $|L:O_p(G)|=p$ and $\cs^*(G)=\{p,|O_p(G):\Z(L)|\,|K|\}$. Since $\gcd(m,n)=1$,
  we must have $O_p(G)=\Z(L)$. Hence $L/\Z(L)\cong C_p$ and $L$ is abelian,
  a contradiction. Thus case (B2) does not occur.
  In case (B3), $\cs^*(G)=\{p^a,p^b|L/L\cap\Z(G)|\}$ by~\cite{DJ}*{Lemma~3.3}. However,
  $a\ge1$ and $b\ge1$ since $p^a=|K:\Z(K)|>1$ and $\cs^*(K)=\{p^b\}$.
  Thus $\gcd(m,n)\ne1$, a contradiction.

  The only remaining possibility is that $G$ has type (B1). In this case
  $K$ and $L$ are abelian, $\Z(G)< L$ and $G/\Z(G)$ is a Frobenius group
  by~\cite{DJ}*{Theorem~A}.
  It follows from the proof of~\cite{DJ}*{Theorem~A} that $K\Z(G)/\Z(G)$ is the
  kernel of $G/\Z(G)$.

  Let  $\overline{\phantom{n}}\colon G\to G/\Z(G)$ be the natural projection,
  let $\Gbar=G/\Z(G)$ and write $\Gbar=\overline{K}\rtimes \overline{L}$.
  By~\cite{DJ}*{Lemma~3.3}, $\cs^*(G)=\{|\overline{K}|,|\overline{L}|\}$
  equals $\{m,n\}$.
  However, $\Gbar$ is a Frobenius group, so for
  $1\ne\overline{k}\in \overline{K}$,
  $C_{\Gbar}(\overline{k})\le \overline{K}$ by~\cite{I}*{Theorem~6.4}.
  This implies $C_{\Gbar}(\overline{k})= \overline{K}$ as $\overline{K}$ is
  abelian. Therefore $1<|\overline{k}^{\Gbar}|\ne|\overline{K}|$, so 
  $|\overline{k}^{\Gbar}|=|\overline{L}|$. Further,
  $|\overline{K}|\equiv 1\pmod {|\overline{L}|}$ by~\cite{I}*{Lemma~6.1},
  so  $m=|\overline{L}|=|L/\Z(G)|$ is less than $n=|\overline{K}|=|K|$
  and $n\equiv 1\pmod m$. Since $\overline{L}$ is abelian, its
  Sylow subgroups are cyclic by~\cite{I}*{Corollary 6.17} and hence
  $L/\Z(G)$ is cyclic.
  This proves one~implication.

  Consider the reverse implication.
  We first prove that $\cs(G)=\cs(\Gbar)$, that is
  $|G:C_G(g)|=|\Gbar:C_{\Gbar}(\overline{g})|$ for all $g\in G$.
  Observe that $L/\Z(G)$ is cyclic implies that $L$ is abelian. 
  Since $\Z(G)\le C_G(g)$ we have $|G:C_G(g)|=|\Gbar:\overline{C_G(g)}|$.
  It suffices to prove that $\overline{C_G(g)}=C_{\Gbar}(\overline{g})$.
  We show $\overline{C_G(g)}\ge C_{\Gbar}(\overline{g})$
  since $\overline{C_G(g)}\le C_{\Gbar}(\overline{g})$ is automatic.
  Consider three cases. (1)~If $g\in\Z(G)$, then
  $\overline{C_G(g)}\le C_{\Gbar}(\overline{g})=\Gbar$.
  (2)~Suppose $g\in K\Z(G)\setminus \Z(G)$, say $g=g_0z$ where $g_0\ne1$.
  Then $C_G(g)=C_G(g_0)\ge K$ (as $K$ is abelian), so
  $\overline{C_G(g)}\ge\overline{K}\ge C_{\Gbar}(\overline{g})$
  by~\cite{I}*{Theorem~6.4(4)}. (3)~If $g\in G\setminus K\Z(G)$, then
  $\overline{g}\not\in\overline{K}$ and so there exists an $x\in G$
  such that $\overline{g}\in\overline{xLx^{-1}}$ as
  $\Gbar$ is a Frobenius group.
  Since $\Z(G)\le xLx^{-1}$, we see that $g\in xLx^{-1}$, that is $g^x\in L$.
  Thus $C_G(g^x)\ge L$  (as $L$ is abelian) and   
  $\overline{C_G(g^{x})}\ge\overline{L}\ge C_{\Gbar}(\overline{g^{x}})$
  by~\cite{I}*{Theorem~6.4(3)}. Hence
  $\overline{C_G(g)}\ge C_{\Gbar}(\overline{g})$.

  We next prove that $\cs(\Gbar)=\{|\overline{K}|,|\overline{L}|\}$.
  View $\Gbar$ as a Frobenius group $\overline{K}\rtimes \overline{L}$ with abelian
  kernel~$\overline{K}$ and abelian complement $\overline{L}$. We show that
  $\cs^*(\overline{K}\rtimes \overline{L})=\{|\overline{K}|,|\overline{L}|\}$.
  Observe that the cosets of $\overline{K}$ in
  $\overline{K}\rtimes \overline{L}$, other than $\overline{L}$,
  form a single conjugacy class~\cite{I}*{p.\,185, 6A.4}, and
  as $\overline{K}$ is abelian the classes
  in $\overline{K}\setminus\{1\}$ all have size $|\overline{L}|$ by~\cite{I}*{Theorem~6.4(4)}. Hence $\cs^*(G)=\cs^*(\Gbar)=\{m,n\}$ where
  $m=|\overline{L}|<n=|\overline{K}|$ satisfy $n\equiv 1\pmod m$. Thus
  $\gcd(m,n)=1$, and the reverse implication holds.
\end{proof}

\begin{proposition}\label{T:coprime}
  Suppose $\cs(G)=\{a_0,a_0+d,\dots,a_0+rd\}$,
  where $a_0,d\ge1$ and $r\ge2$.
  Then $a_0=d=1$, $r=2$, so $\cs(G)=\{1,2,3\}$, and $G\cong G_n\times C$ where
  \begin{equation}\label{E:G}
    G_n=\langle a,b\mid a^{2^n}=b^3=1,\;b^a=b^{-1}\rangle
    \qquad\textup{}
  \end{equation}
  and $C$ is abelian. Clearly $G/\Z(G)\cong \Sym_3$ where $\Z(G)=\langle a^2\rangle\times C$.
  Conversely, $\cs(G_n\times C)=\{1,2,3\}$ for all $n\ge1$ and
  all abelian groups $C$.
\end{proposition}%theorem}

\begin{proof}%[Proof of Proposition~\ref{T:coprime}]
  Now $1\in G$ implies $a_0=1$. Hence consecutive terms of
  $\cs^*(G)=\{1+d,\dots,1+rd\}$ are
  coprime.   It follows from \cite{BGC}*{Corollary~2, p.\,260} that $r\le2$,
  and hence $r=2$. Thus $\cs(G)=\{1,1+d,1+2d\}$.
  Suppose that $a,b\in G$ where $|a^G|=1+d$ and $|b^G|=1+2d$.
  By~\cite{BGC}*{Theorem, p.\,255}, $G=NH$ where $N=C_G(a)$, $H=C_G(b)$
  are abelian, $H\cap N=\Z(G)$, and $G/\Z(G)$ is a Frobenius group with
  kernel $N/\Z(G)$ and complement $H\Z(G)/\Z(G)$.
  Since $\cs(X\times C)=\cs(X)$ for all abelian $C$, we may assume that
  $G$ has no abelian direct factors, so Theorem~\ref{T:KL} applies.
  However, $1+2d\equiv 1 \mod(1+d)$ holds by Theorem~\ref{T:KL}.
  Thus $2d=k(1+d)$ for some integer $k$, so $k=1$ and $d=1$.
  Therefore $\cs(G)=\{1,2,3\}$ and the structure of $G$ is determined
  by \cite{B92}*{Theorem 1}. Paraphrasing this result, $G$ has the
  form~\eqref{E:G}.

  Conversely, the elements of $G=G_n\times C$, see~\eqref{E:G}, have a normal
  form $a^ib^jc$ where $0\le i<2^n$, $0\le j<3$ and $c\in C$. 
  A simple calculation shows that $|(a^ib^jc)^G|=3$ if $i\ne 0$, 
  $|(b^jc)^G|=2$ if $j\ne0$, and $|c^G|=1$. Thus $\cs(G)=\{1,2,3\}$.
  Let $N=\langle a^2\rangle\times C$.
  Since $G/N\cong\Sym_3$ and $\Z(\Sym_3)=1$, an easy calculation shows that
  $\Z(G)=N$.
\end{proof}

Having deduced that $d=1$ in the above proof, we could
also invoke~\cite{B92}*{Theorem~2} with $p^b=3$. However, \cite{B92}*{Theorem 1}
required less work.

\section{Arithmetic progressions involving two primes}\label{s:ap}

In Section~\ref{s:conjsizes} we saw that group theory strongly constrains
when $\cs(G)$ can be an arithmetical progression. This section explores the
extent to which number theory alone imposes constraints.

Most research concerning arithmetical progressions and primes falls in two
main areas. The first concerns sets of primes containing arbitrarily long
arithmetic progressions~\cite{GT}.
The second concerns quantifying the distribution of smooth numbers in arithmetic
progressions, for example~\cite{BP}. (A positive integer is called
\emph{$y$-smooth} if all its prime factors are at most $y$.) This section
is motivated by the latter.

We say that an arithmetical progression $a_0,a_1,\dots,a_r$ \emph{involves} at most
two primes if $|\bigcup_{i=0}^r \pi(a_i)|\le2$ where
$\pi(a_i)$ denotes the set of prime divisors of $a_i$.

Let $a_i=a_0+id$
for $i=0,1,\dots,k$ and set $\delta:=\gcd(a_0,d)$. For all $i\ge1$ we have,
\[
  \gcd(a_{i-1},a_i)=\delta\quad\textup{and}\quad
  2a_i=a_{i-1}+a_{i+1}.
\]
An arithmetic progression is called \emph{primitive} if $\delta=\gcd(a_0,d)=1$.

Our Lemma~\ref{L:NT} relies on an easy number-theoretic lemma
of John Thompson.

\begin{lemma}\label{T:M}\textup{\cite{T}*{Lemma~3}}
  Let $p$ be an odd prime.
  Then the only solutions to $p^m=2^n\pm1$ have $m=1$ and $p$ a Fermat
  or Mersenne prime, or $3^2=2^3+1$.
\end{lemma}

\begin{lemma}\label{L:NT}
  Suppose $k\ge2$ and $(a_0, a_1, \dots ,a_k)$ is a primitive arithmetic
  progression involving at most two primes and $1\le a_0<a_1$. If $k\ge3$, the sequence
  must be $(1,2,3,4)$. If $k=2$, then $a_0\in\{1,2\}$ and the sequence involves
  precisely two primes, say $p$ and~$q$. Moreover, $(a_0, a_1, a_2)$ equals
  one of the following:
  \begin{enumerate}[{\rm (i)}]
  \item $(1,2^\alpha,2^{\alpha+1}-1)$ where $2^{\alpha+1}-1$ is a Mersenne
    prime, or
  \item $(1,p^\alpha,q^\beta)$ where $p>2$, $q^\beta\equiv1\pmod 4$
    and $1+q^\beta=2p^\alpha$, or
  \item $(2,q, 2^{\alpha+1})$ where $q=2^\alpha+1$ is a
    Fermat prime, or
  \item $(2,3^2,2^4)$.
  \end{enumerate}
\end{lemma}

\begin{proof}
  We first classify the arithmetic sequences $(a_0,a_1,a_2)$ with three
  terms.  Write $a_i=p^{\alpha_i}q^{\beta_i}$ for $i=0,1,2$. Let $d=a_1-a_0$.
  Since $\delta=1$, we have
  \[
  \min\{\alpha_0,\alpha_1\}=\min\{\alpha_1,\alpha_2\}=
  \min\{\beta_0,\beta_1\}=\min\{\beta_1,\beta_2\}=0.
  \]

  {\sc Case $a_0=1$.} Here $\alpha_0=\beta_0=0$. Since $a_1\ge2$, one of
  $\alpha_1$ or $\beta_1$ is positive. Interchanging $p$ and $q$ if necessary,
  assume that $\alpha_1>0$. This forces $\alpha_2=0$, so $a_2=q^{\beta_2}$.
  Hence in turn $\beta_2>0$, so $\beta_1=0$ and
  $(a_0,a_1,a_2)=(1,p^{\alpha_1},q^{\beta_2})$. Thus $1+q^{\beta_2}=2p^{\alpha_1}$
  and so $q$ is odd. If $p>2$, then $1+q^{\beta_2}\equiv 2\pmod4$ shows~(ii) holds.
  If $p=2$, then $1+q^{\beta_2}=2^{\alpha_1+1}$ implies by Thompson's Lemma
  that $\beta_2=1$, and hence $q=2^{\alpha_1+1}-1$ is a Mersenne prime
  (so $\alpha_1+1$ must be prime). This is case~(i).

  {\sc Case $a_0=2$.} Take $p=2$. Thus $\alpha_0=1$ and $\beta_0=0$. Therefore
  the arithmetical sequence
  $(a_0,a_1,a_2)=(2,2^0q^{\beta_1},2^{\alpha_2})$ satisfies
  $2+2^{\alpha_2}=2q^{\beta_1}$, that is $1+2^{\alpha_2-1}=q^{\beta_1}$.
  If $\alpha_2-1\ge2$ and $\beta_1\ge2$, then
  this equation is $1+2^3=3^2$ by Thompson's Lemma. Hence
  $(a_0,a_1,a_2)=(2,3^2,2^4)$ and case~(iv) holds. Suppose now that
  $\alpha_2-1\in\{0,1\}$. Then $1+2^{\alpha_2-1}$ equals $2$ or~$3$. However,
  $q\ne p$, so the only possibility is $(a_0,a_1,a_2)=(2,3,2^2)$ and
  case~(iii) holds. Finally,
  suppose that $\beta_1=1$. Then $1+2^{\alpha_2-1}=q$ is a Fermat prime.
  This is case~(iii), and $\alpha_2-1$ is a power of $2$. A specific
  instance is $(a_0,a_1,a_2)=(2,3,2^2)$ which extends to $(1,2,3,4)$.

  {\sc Case $a_0\ge3$.} It is not possible that $\alpha_0>0$ and
  $\beta_0>0$. Otherwise $\alpha_1=\beta_1=0$, so $a_1=1$ and
  $3\le a_0\le a_1=1$, a contradiction.
  Hence one of $\alpha_0$ and $\beta_0$ is zero. Swapping $p$ and $q$
  if necessary, we may assume that $\beta_0=0$. Arguing as above,
  we have $(a_0,a_1,a_2)=(p^{\alpha_0},q^{\beta_1},p^{\alpha_2})$
  where $p^{\alpha_0}+p^{\alpha_2}=2q^{\beta_1}$. Since $0<\alpha_0<\alpha_2$,
  $p^{\alpha_0}$ divides the left-side. Since $p\ne q$ this implies that
  $p=2$. However $a_0\ge3$ shows $\alpha_0\ge2$ and so $4\mid 2q^{\beta_1}$,
  a contradiction. Thus this case never occurs.

  We have now classified the arithmetic progressions with precisely
  three terms involving
  at most two primes. If $(a_0,a_1,a_2,a_3)$ involves at most two primes,
  then it follows from parts (i)--(iv) that $(a_0,a_1,a_2)=(1,2,3)$ and
  hence $(a_0,a_1,a_2,a_3)=(1,2,3,4)$. Finally, $k\le3$ as $(1,2,3,4,5)$
  involves more than two primes.
\end{proof}

\begin{remark}\label{R6}
%  \begin{enumerate}[{\rm (a)}]
  (a)~The primitive arithmetic progressions with first term
  $a_0=1$ are therefore
  the sequences in (i) and (ii) and $(1,2,3,4)$. The work of~\cite{BGC}
  excludes $(1,2,3,4)$ from occurring as $\cs(G)$, and Proposition~\ref{T:coprime}
  shows that only $(1,2,3)$ arises. Thus Lemma~\ref{L:NT} illustrates the
  limitations of using number theory only.
  \vskip1mm
  (b)~Dividing each term of a non-primitive arithmetic progression
  by $\delta=\gcd(a_0,d)$ gives a primitive one. Thus
  all non-primitive arithmetic progressions $(a_0, a_1, a_2)$ involving
  distinct primes $p,q$ can be classified using Lemma~\ref{L:NT} by multiplying
  by $\delta=p^\alpha q^\beta>1$. We now consider the
  non-primitive arithmetic progression $(2,4,6)$.
%  \end{enumerate}
\end{remark}

\section{Examples of groups with \texorpdfstring{$\cs^*(G)=\{2,4,6\}$}{}}
\label{sec:ex}

In this section we consider groups of the form $3.A$ where $A$ is nilpotent.
In particular, we assume that $G=AB$ satisfies:
\begin{enumerate}[{\rm (i)}]
  \item $|A|$ is a power of 2;
  \item $C_2\cong A'$;
  \item $C_3\cong B\lhdeq G$; and
  \item $\Z(A)< C_A(B)<A$.
\end{enumerate}

\begin{remark}\label{R246}
We used {\sc Magma}~\cite{Bosma} to find many groups $G$, with
no (non-trivial) abelian direct factors, satisfying (i)--(iv).
Our {\sc Magma} program found that there are 170 such groups whose
order divides $2^8\cdot 3$.
\end{remark}

\begin{lemma}\label{L:construct}
  If $G=AB$ satisfies \textup{(i)--(iv)} above, then
  $\cs^*(G)=\{2,4,6\}$.
\end{lemma}

\begin{proof}
  Each element of $G$ can be written uniquely as $ab$ where
  $a\in A$ and $b\in B$. Fix $a$ and $b$, and consider the conjugacy class
  $(ab)^G=\{(ab)^{a'b'}\mid a'\in A,\, b'\in B\}$:
  \[
    (ab)^{a'b'}=a^{a'b'}b^{a'b'}=a[a,a'b']b^{a'b'}=a[a,b'][a,a']^{b'}b^{a'b'}.
  \]
  As $A'$ is normal in the 2-group $A$ and $|A'|=2$, we have $A'\le\Z(A)$.
  Further, as $C_A(B)<A$ and $B\cong C_3$, $[A,B]=B$. Also,
  $[a,a']\in A'\le\Z(A)\le C_A(B)$ so
  this expression can be simplified as follows:
  \begin{equation}\label{E:ab}
  (ab)^{a'b'}=a[a,b'][a,a']b^{a'}=a[a,a'][a,b']b^{a'}\quad
  \textup{ where $a[a,a']\in A, [a,b']b^{a'}\in B$}.
  \end{equation}
  Suppose that $a'\in A$ and $b'\in B$ vary. Then we have
  \def\k{\kern-1.5pt}
  \[
    [a,A]\k=\k\begin{cases}\{1\}&\textup{if $a\k\in\k \Z(A)$,}\\
                       A'&\textup{otherwise,}\end{cases}\\\;
    [a,B]\k=\k\begin{cases}\{1\}&\textup{if $a\k\in\k C_A(B)$,}\\
                       B&\textup{otherwise,}\end{cases}\;
    b^A\k=\k\begin{cases}\{1\}&\textup{if $b=1$,}\\
                     \{b,b^2\}&\textup{otherwise.}\end{cases}
  \]
  The cardinalities of $[a,A]$, $[a,B]$ and $b^A$ are 1,2 or 1,3 or 1,2.
  This shows that the size of a conjugacy class lies in~$\{1,2,3,4,6,12\}$.
  Since $\{[a,b']b^{a'}\mid b'\in B, a'\in A\}$ equals $B$ when
  $a\not\in C_A(B)$, there are no classes of size 12.  Observe that
  $[a,B]=B$ precisely when $a\not\in C_A(B)$ and in this case
  $a\not\in \Z(A)$ so $[a,A]=A'$. This shows that a class size of 3 is also
  not possible. Thus $\cs(G)\subseteq\{1,2,4,6\}$

  Conversely, we show that class sizes 2, 4, 6 do arise.
  Let $B=\langle w\rangle$. Now $A$ acts non-trivially on $B$ since
  $C_A(B)<A$. Thus $|w^G|=2$. Choose $a\in A\setminus C_A(B)$.
  Then $|[a,B]|=3$. As $a\not\in\Z(A)$, we see that $|[A,a]|=2$.
  Thus $|a^G|$ is divisible by~6, so $|a^G|=6$. Finally, $|(aw)^G|=4$ for
  $a\in C_A(B)\setminus\Z(A)\ne\emptyset$.
\end{proof}

In Theorem~\ref{T246} we prove the converse of Lemma~\ref{L:construct},
that is, we prove that a group~$G$
satisfying $\cs^*(G)=\{2,4,6\}$ must satisfy conditions (i)--(iv) above.

\section{Proof of Theorem~\ref{T246}}\label{sec:T246}

Lemma~\ref{L:construct} gives a class of groups $G$
with $\cs^*(G)=\{2,4,6\}$. This is the easy part of the proof of
Theorem~\ref{T246}. In this section, we give a detailed proof that
these are the only examples. The following lemma
paraphrases \cite{CH}*{Proposition 4}.

\begin{lemma}
  Suppose that $p$ is a prime divisor of $|G|$
  and $\cs^*(G)=\{n_1,\dots,n_r\}$. Then
  $p\nmid n_1\cdots n_r$ if and only if a Sylow $p$-subgroup
  of~$G$ is central.
\end{lemma}

Thus it follows from Burnside's $p$-complement theorem that
$p\mid G$ and $p\nmid n_1\cdots n_r$ implies $G$ has a non-trivial
abelian direct factor. We henceforth assume that $G$ has no non-trivial
abelian direct factor: clearly $\cs(G)=\cs(G\times A)$ for $A$ abelian.
Thus for us, the prime divisors of $n_1\cdots n_r$ coincide with the prime
divisors of~$G$.

\begin{proof}[Proof of Theorem~\ref{T246}]
  Let $G$ be a finite group with $\cs^*(G)=\{2,4,6\}$ and no abelian
  direct factors. By the preceding argument, $G$ is a $\{2,3\}$-group.
  Since $3\not\in\cs(G)$, it follows that~$G$ is not nilpotent, so $F(G)<G$.

  A result of Gasch\"utz~\cite{Huppert}*{Satz III.4.5} says that
  $F(G)/\Phi(G)$ is a direct product of abelian minimal normal subgroups
  of $G/\Phi(G)$. Hence the group $\Gbar := G/\Phi(G)$
  may be written as $(P\times Q)\rtimes R$ where $F=F(G)/\Phi(G)$
  has Sylow $2$-subgroup $P$, Sylow $3$-subgroup $Q$, and both
  are elementary abelian and $F=P\times Q$. Now $\Gbar/F$ acts
  faithfully on $F$ as $F(G)/\Phi(G)=F$ and
  $C_{\Gbar}(F)\le F$ by~\cite{Huppert}*{III\;Satz~4.2}.
  Hence $R$ acts linearly (perhaps not faithfully)
  and completely reducibly on both $P$ and $Q$.

  Our argument is similar in parts to~\cite{BGP}*{pp.\,4--6} although
  our notation differs. Since $\Gbar:=G/\Phi(G)$, $\Phi(\Gbar)$ is trivial.
  We will write $\Gbar=F\rtimes R$ where
  $F=P\times Q$ are elementary abelian Sylow $2$- and
  Sylow $3$-subgroups of $F$. 
  As $\Gbar$-conjugacy class sizes are divisors of
  the $G$-conjugacy class sizes, we have
  $\cs^*(\Gbar)\subseteq\{2,3,4,6\}$.

We split the proof into two cases depending on how $R$ acts on $Q$.

\smallskip\noindent
{\bf Case A.} $R$ acts non-trivially on $Q$.

\smallskip%\noindent
{\bf Step A1.} \emph{There exists a $2$-element $x\in R_2\setminus C_{R_2}(Q)$, for $R_2$ a Sylow $2$-subgroup of $R$, such that $U:= [Q,\langle x\rangle]$ has order $3$ and is inverted by $x$, and $Q=C_Q(x)\times U$.}

As $R$ acts non-trivially on $Q$, $C_R(Q)$ is a proper normal subgroup
of $R$. To prove the existence of a suitable element $x$, let $C_R(Q) < L
\lhdeq R$ such that $L/C_R(Q)$ is a minimal normal subgroup of
$R/C_R(Q)$. Since $R$ is completely reducible on $Q$ (regarded as a
vector space over $\mathbb{F}_3$), it follows that $L/C_R(Q)$ is an
elementary abelian $2$-group. Choose $x\in L\setminus C_R(Q)$. Replacing
$x$ by an odd power of itself, we can (and will) assume that $x$ is a
$2$-element. By definition, $x$ acts non-trivially on $Q$ and it lies in
some Sylow $2$-subgroup $R_2$ of $R$. Thus $x\in R_2\setminus
C_{R_2}(Q)$.

As $x$ acts non-trivially on $Q$, we have $U=[Q,\langle x\rangle]\ne 1$ and
$C_Q(x)\ne Q$. Also $Q=C_Q(x)\times U$ since $x$ has order coprime to~$3$.
The $\Gbar$-conjugacy class size of $x$,
namely $|x^{\Gbar}|=|\Gbar:C_{\Gbar}(x)|$, is
divisible by $|Q:C_Q(x)|=|U|=3^u$ for some $u\geq 1$.  However, as
$|\Gbar:C_{\Gbar}(x)|$ is a divisor of one of $2,4,6$, it follows that
$|\Gbar:C_{\Gbar}(x)|=|Q:C_Q(x)|=|U|=3^u=3$. Thus $U$ has order~$3$
and is inverted by $x$.

\smallskip%\noindent
{\bf Step A2.} \emph{Let $x$ be as in Step~\textup{A1}. Then $R/C_R(Q)=\la x C_R(Q)\ra\cong C_2$. Further, there is a normal subgroup $H$ of $G$ containing $\Phi(G)$ such that $G/H\cong \Sym_3$, and $\Hbar := H/\Phi(G) = (P \times C_Q(x)).C_R(Q)$.}

It follows from Step A1 that $x$ induces a linear transformation
of~$Q$ with determinant $-1$, and the same holds for all
$y\in R_2\setminus C_{R_2}(Q)$. In particular, $y^2\in C_{R_2}(Q)$ for all
such $y$, and hence $R_2/C_{R_2}(Q)$ is an elementary abelian
$2$-group. Further the product $xy$ of two such elements must have
determinant 1, and so $xy$ must lie in $C_{R_2}(Q)$. This
implies that $R_2/C_{R_2}(Q)\cong C_2$.

Now $R_2C_R(Q)/C_R(Q)\cong R_2/C_{R_2}(Q)\cong C_2$, and hence a Sylow
$2$-subgroup of $R/C_R(Q)$ has order 2. Thus $L/C_R(Q)\cong C_2$ (the
minimal normal subgroup in the proof of Step A1) and lies in all Sylow
$2$-subgroups of $R/C_R(Q)$, so $L=R_2C_R(Q)$. This holds for all
minimal normal subgroups of $R/C_R(Q)$ and hence $L/C_R(Q)$ is the
unique minimal normal subgroup of $R/C_R(Q)$. Since $R/C_R(Q)$ is a
$\{2,3\}$-group it follows that $R/C_R(Q)=\langle xC_R(Q)\rangle\cong C_2$.

Since each of $C_Q(x)$ and $U=[Q,\langle x\rangle]$ is invariant under
$F, C_R(Q)$ and $x$, it follows that $C_Q(x)$ and $U$ are normal
subgroups of $\Gbar$. Also $C_R(Q)=C_R(U)$ is centralised by $Q$ and~$R$,
and hence by $\Hbar = (P \times C_Q(x)).C_R(Q)\lhdeq \Gbar$. The
quotient is generated by $U\Hbar/\Hbar\cong U$ and $x\Hbar$, and so is
nonabelian of order 6. Let $H$ be the full preimage of $\Hbar$.
Then $G/H\cong\Gbar/\Hbar\cong \Sym_3$ as claimed.

\smallskip%\noindent
{\bf Step A3.} \emph{Let $\pi:G\rightarrow \Sym_3$ be the natural projection
  with kernel $H$ as in Step~\textup{A2}. Let
  $T:=\{ g\in G \mid \pi(g)\textup{ has order~$2$}\}$.
  Then $\Z(G)=C_H(T)\leq H$,  $H/\Z(G)$ is an elementary abelian
  $2$-group, and $H\ne \Z(G)$. Also $|H:C_H(g)|=2$ for each $g\in T$.}

Let $a\in T$, so $Ha=\pi(a)$ has order $2$ in $G/H\cong\Sym_3$. Then
$Ha$ lies in an $\Sym_3$-conjugacy class of size $3$, so $3$ divides
the class size $|a^G|$. Since $\cs(G)=\{2,4,6\}$ it follows that
$|a^G|=6$ and hence $|H:C_H(a)|=2$. The natural map $H\rightarrow
\prod_{a\in T} H/C_H(a)$ which sends each $h\in H$ to the $|T|$-tuple
with $a$-entry $C_H(a)h$ is a group homomorphism from $H$ to an
elementary abelian $2$-group with kernel $C_H(T)$. In particular
$H/C_H(T)$ is an elementary abelian $2$-group.

We now show that $C_H(T)\leq \Z(G)$. Let $g\in G$. We show that
$C_H(T)$ centralises~$g$. (i)~If $\pi(g)$ has order $1$, then $g\in H$.
Thus for $a\in T$, we have $ga, a^{-1}\in T$, so $C_H(T)$
centralises $ga$ and $a^{-1}$ and hence also $gaa^{-1}=g$.  (ii)~If
$\pi(g)$ has order $2$, then $g\in T$ and by definition, $g$
centralises $C_H(T)$. (iii)~If $\pi(g)$ has order $3$ then
$\pi(g)=\pi(a)\pi(b)$ for two elements $\pi(a),\pi(b)$ of $\Sym_3$ or
order $2$. Thus $g=hab$ for some $h\in H$, $\pi(ha)=\pi(a)$, and so
$ha, b\in T$. Therefore $C_H(T)$ centralises $ha$ and $b$ and hence also
centralises $hab=g$. This proves that $C_H(T)\leq \Z(G)$.

Since $\Sym_3$ has trivial centre it follows that $\Z(G)\leq H$. If $g\in \Z(G)$, then we just showed that $g\in H$, and since $g$ is in $\Z(G)$ it must in particular centralise $T$, so $g\in C_H(T)$. Thus $\Z(G)=C_H(T)$. By the first paragraph of this argument, $H$ does not centralise the element $a$, and so $H\ne \Z(G)$.

\smallskip%\noindent
{\bf Step A4.} \emph{Let $A$ be a Sylow $2$-subgroup of $G$, and let $B$ be a Sylow $3$-subgroup of $G$. Then $B$ is abelian and normal in $G$. Moreover $\Z(G)=(A\cap \Z(G))\times (B\cap H)$, and $F(G)=(A\cap H)\times B$ has index $2$ in $G$. Also  $A$ is nonabelian and $A'\cong C_2$.}
  
 By Step A3, $H$ is an extension of an abelian subgroup $\Z(G)$ by an
 abelian group $H/\Z(G)$, and hence $H$ is nilpotent so $H\leq
 F(G)$. Then since $G/H\cong \Sym_3$ is not nilpotent, $F(G)$ has
 index 6 or 2 in $G$.
   That is, $F(G)$ equals $H$ or $N$ where $N$ is the
   normal subgroup of $G$ satisfying $F(G)\le N$ and $|G:N|=2$.
   By Step~A3 the Sylow 3-subgroup $H_3$ of $H$ lies in $\Z(G)$. If $F(G)=H$,
   then $\overline{F(G)}=\overline{H}$ has Sylow 3-subgroup
   $Q=\overline{H_3}\le\overline{\Z(G)}\le\Z(\overline{G})$.
   This contradicts Step~A1. Hence
   $F(G)=N$ and the Sylow $3$-subgroup $B$ of $G$ lies in $F(G)$.
   Since $B \leq N$ and $N=F(G)$, it follows that $B$ is the unique
   Sylow 3-subgroup of $G$, and $B$ is normal in $G$. Moreover,
   $B\cap H$ has index $3$ in $B$, and by Step A3, $B\cap H\leq \Z(G)$.
   Thus $B$ is an extension of a central subgroup $B\cap H$ by a cyclic group,
   and hence $B$ is abelian.

Recall that $A$ is a Sylow $2$-subgroup of $G$.  We have  $F(G)=(A\cap H)\times B$ with $A\cap H$ of index $2$ in $A$. Hence $H=(A\cap H)\times (B\cap H)$, and since $B\cap H\leq \Z(G)< H$, we have $\Z(G) = (A\cap \Z(G))\times (B\cap H)$.  
It was shown in the proof of Step A3 that, each $a\in A\setminus(A\cap H)$
has $|H:C_H(a)|=2$. Thus $a\not\in \Z(A)$ so $A$ is nonabelian and
$\Z(A)\leq A\cap H$. Now each $a\in A\setminus (A\cap H)$ lies in the set $T$ of Step A3. %% We showed in the proof of Step A3 that $|H:C_H(a)|=2$.
Since $H=(A\cap H)\times (B\cap H)$ and $B\cap H$ is central, it follows that $C_{A\cap H}(a)$ has index $2$ in $A\cap H$. Hence $A'\cong C_2$ by~\cite{BGP}*{Lemma 1.1}.

\smallskip%\noindent
{\bf Step A5.} \emph{Using the notation of Step~\textup{A4}, $G=AB$ where 
 $A$ is a $2$-group, $B\lhdeq G$ has order $3$,
   $|A'|=2$, and $\Z(A)<C_A(B)<A$.}

  First we observe that, for $a, g\in G$, $a^g=a[a,g]$, and hence the
  conjugacy class $a^G$ equals $a[a,G]$, so $|a^G|=|[a,G]|\leq 6$. 

By Step A4, $B\lhdeq G$, $B$ is abelian, and $B\cap H\leq \Z(G)$. We
show that $B$ is cyclic.  Suppose to the contrary that $B$ has
rank~$s>1$. By the theory of $\mathbb{Z}$-modules there exist
decompositions $B=C_{n_1}\times C_{n_2}\times\cdots\times C_{n_s}$ and
$B\cap H=C_{n_1/3}\times M$ where $M=C_{n_2}\times\cdots\times
C_{n_s}\ne 1$. Then each element $h\in B$ can be written uniquely as
$h_1^km$ where $C_{n_1}=\langle h_1\rangle$ and $m\in M\le\Z(G)$.
Choose $g\in G\setminus F(G)$. Since $B\lhdeq G$, the image
$(h_1)^g\in B$ and hence $(h_1)^g=h_1^km$ for some (unique)
non-negative $k<|h_1|$ and $m\in M$. If $m=1$, then as $G=\langle
F(G), g\rangle$ and $F(G)$ centralises $B$ (by Step A4), it follows
that $\langle h_1\rangle \lhdeq G$, and hence $G=(A\langle
h_1\rangle)\times M$ has a non-trivial abelian direct factor. This is a
contradiction. Hence $m\ne1$. Now $g^2\in F(G)$ (since $|G:F(G)|=2$),
and since $B$ is abelian and $M\leq \Z(G)$, we have
\[
  h_1=(h_1)^{g^2}=(h_1^km)^{g}= (h_1^g)^k m^g = (h_1^km)^k m=
    h_1^{k^2}m^{k+1}.
\] 
Thus $k^2\equiv 1\pmod{n_1}$ and $m^{k+1}=1$. Since $m$ is a
non-trivial $3$-element it follows that $3$ divides $k+1$, and
therefore $3$ does not divide $k-1$. However $n_1\geq3$ is a power of $3$
that divides
$k^2-1$, and hence $n_1$ divides $k+1$.  Since $0\leq k<n_1$ this
implies that $k=n_1-1$ and $(h_1)^g=h_1^{-1}m$ with $m^{n_1}=1$.  Thus
$[h_1,g]=h_1^{-1}h_1^g=h_1^{-2}m$, and similarly, for each $\ell$,
\[
[h_1^\ell,g]=h_1^{-\ell}(h_1^\ell)^g=h_1^{-\ell}(h_1^g)^\ell=h_1^{-\ell}(h_1^{-1}m)^\ell=
h_1^{-2\ell}m^\ell
\]
and since $|h_1|=n_1$ is coprime to $2$, it follows that $|[\langle
h_1\rangle,g]|\geq n_1$.  However, as we observed above,
$|g^G|=|[G,g]|\geq |[\langle h_1\rangle,g]|\geq n_1$, and since
$|g^G|\leq 6$ and $n_1$ is a power of $3$, we conclude that
$n_1=3$. Also, since $m^{n_1}=1$ we have $|h_1|=|m|=3$ and
$(h_1)^g=h_1^2m$. Now the set $\{h_1,h_1^{2}m\}$ is invariant under
$g$ and is centralised by $F(G)$, and so is $G$-invariant. Hence
$B_0:=\langle h_1, m\rangle \cong C_3\times C_3$ is normal in $G$.
Now $B_0$ has exactly four subgroups of order $3$, and the two
subgroups $\langle m\rangle, \langle h_1 m\rangle$ are $G$-invariant
(recall that $m\in \Z(G)$). Therefore we may replace $C_{n_1}$ by
$\langle h_1m\rangle$ in the decomposition for $B$, and since
$(h_1m)^g=h_1^2m^2=(h_1m)^2$, we obtain $G=(A\langle
h_1m\rangle)\times M$ with a non-trivial abelian direct factor, which
is a contradiction. Thus we have proved that $B$ has one
generator, that is, $B$ is cyclic, say $B=\langle b\rangle$.

We now prove $|B|=3$. Recall that $G=AB$ where $B\lhdeq G$ is a
cyclic $3$-subgroup and $A$ is a Sylow $2$-subgroup with $|A'|=2$ by
Step~A4.  By Step A3, there exists $a\in A\cap H$ such that $a$
inverts $bH$, and since $B\lhdeq G$, this means that $b^a=b^k\ne b$.
Since $a^2\in A\cap F(G)=A\cap H$ centralises $B$, we have
$k^2\equiv 1\mod {|B|}$.  Since $\Aut(B)$ is cyclic of order
$2|B|/3$, it has a unique involution whence $k\equiv -1\mod{|B|}$.
Thus $[b,a]=b^{-1}b^a=b^{k-1}=b^{-2}$.  Arguing as in the
previous paragraph we have $6\geq |a^G|=|[G,a]|\geq |[B,a]|\geq|B|$,
and hence $|B|=3$ as desired.

  Finally, we show that $\Z(A)<C_A(B)$. If $\Z(A)=C_A(B)$, then it follows
  from the proof of Lemma~\ref{L:construct} that $G$ has no conjugacy
  classes of size $4$, a contradiction. (Using the notation of
  Lemma~\ref{L:construct}, if $|(ab)^G|=4$, then $|[a,B]|=1$ so
  $a\in C_B(A)=\Z(A)$ and hence $|[a,A]|=1$, so $|(ab)^G|\le 2$.) 
  
  This completes the proof of Step A4, and hence shows that $G$ has the
  required properties of Theorem~\ref{T246}. This completes the proof of Case~A.

\smallskip\noindent
{\bf Case B.} $R$ acts trivially on $Q$.

Since $G$ is not nilpotent, $R$ acts non-trivially on $P$.
As $R$ acts completely reducibility on $P$, it acts non-trivially on some
irreducible subspace $V$ of $P$.

\smallskip%\noindent
{\bf Step B1.} \emph{Let $V$ be an irreducible subspace of $P$ on which
  $R$ acts non-trivially. Then $|V|=4$ and $R/C_R(V)\cong C_3$.}

We show first that $|V|=4$. As $|P|$ is a power of 2, 
some $R$-orbit on non-zero elements of $V$ has odd size greater than 1.
However $\cs(G)=\{2,4,6\}$ implies that this orbit has size $3$,
say $\{a,b,c\}$. Using multiplicative notation, the product $abc$ is then fixed by $R$, and hence $abc=1$.
By minimality, $V=\la a,b,c\ra = \la a,b\ra$ and hence $|V|=4$. Now
$R/C_R(V)\le \Aut(V)=\textup{GL}_2(2)\cong\Sym_3$, and since $R$ is
irreducible on $V$,  $R/C_R(V)$ equals $C_3$ or~$\Sym_3$.

We will show $R/C_R(V)=C_3$. Suppose not. Then $R/C_R(V)=\Sym_3$.
Let $U$ an $R$-invariant complement to $V$ in $P$, so $P=U\times V$ by complete
reducibility and $U\lhdeq \Gbar$. 
Observe that $F=P\times Q$,  and so
\[
FC_R(V)/((U\times Q)C_R(V))\cong F/(U\times Q)\cong (P\times Q)/(U\times Q)\cong P/U \cong V
\]
is normal in  $\Gbar/((U\times Q)C_R(V))$. The quotient is isomorphic to $\Gbar/FC_R(V)\cong R/C_R(V)\cong \Sym_3$.
Moreover, $\Gbar/((U\times Q)C_R(V))\cong  V\rtimes\Sym_3\cong\Sym_4$.
However, the $3$-cycles of $\Sym_4$ form a conjugacy class of size $8$,
and this conjugacy class size must divide some $G$-conjugacy class size.
This is a contradiction since   $\cs(G)=\{2,4,6\}$. Hence $R/C_R(V)\cong C_3$
as claimed.

\smallskip%\noindent
{\bf Step B2.} \emph{Suppose that $V$ is as in Step B1. Then $R\cong C_3$,
 $F=\Z(\Gbar)\times V$, and $\Gbar/\Z(\Gbar)\cong V\rtimes R\cong A_4$.}

Recall that $R$ centralises $Q$. By Step B1, $R/C_R(V)\cong C_3$ for each non-central minimal normal subgroup $V$ of $F$. Since $R$ is faithful on $F$, it follows that $R$ is an elementary abelian $3$-group. Let $1\ne x\in R$. Then $C_{\Gbar}(x)$ contains $R$ and $Q$, and so 
\[
  |\Gbar:C_{\Gbar}(x)| = |\Gbar:RC_F(x)| = |F:C_F(x)|=|P:C_{P}(x)|
\]
which is a power of $2$. Since $x\ne 1$ and $R$ is faithful on $P$,
$x$ acts non-trivially on some minimal normal subgroup $V$ of
$P$. Using the multiplicative notation of Step~B1, we may write
$V\setminus\{1\}$ as $\{a,b,c\}$ and assume that $a^x=b$, $b^x=c$,
$c^x=a$. It is straightforward to show that the elements $x$, $x^a$, $x^b$,
$x^c$ are pairwise distinct: for example, if $x^a=x^b$ then $x$
centralises $ab=c\in V$, which is not the case. Hence the
$\Gbar$-class of $x$ has size at least 4, and since the size is a
$2$-power, it must be $4$. Thus $P = C_P(x)\times V$.  
Since, for any given $g\in \Gbar$,  $x^g=x^h$
for some $h\in P$, we have $C_P(x)^g=C_{P^g}(x^g)=C_P(x^h)$. Now $y\in
C_P(x^h)$ if and only if $yh^{-1}xh=h^{-1}xhy$, and since $h, y$ commute (because $P$ is abelian) this is equivalent to $yx=xy$, that is, $y\in C_P(x)$. Thus  $C_P(x^h)=C_P(x)$, and hence
$C_P(x)\lhdeq \Gbar$.

We now show that $R$ acts trivially on $C_P(x)$.  Suppose not.  Then
$R$ acts non-trivially on $C_P(x)$, so there exists a non-central
minimal normal subgroup $W$ of $\Gbar$ contained in~$C_P(x)$.  If
$C_R(V)=C_R(W)$ then $\la C_R(V),x\ra=R$ would centralise $W$, which
is not the case, so $C_R(V), C_R(W)$ are distinct proper subgroups of
$R$ and hence there exists $y\in R\setminus(C_R(V)\cup C_R(W))$. This
means that $y$ acts non-trivially on both $V$ and $W$, and hence
$|F:C_F(y)|\geq |VW|=16$, implying that the $\Gbar$-conjugacy class
size of $y$ is at least 16, contradiction.  Hence $R$ centralises
$C_P(x)$, and since $R$ centralises $Q$ and acts faithfully on $F$,
it follows that $R\cong C_3$ and $C_P(x)=C_P(R)=C_P(\Gbar)$. This implies
that $C_F(\Gbar)=\Z(\Gbar)$ has index 4 in $F$, and $F=\Z(\Gbar)\times V$.
Hence $\Gbar/\Z(\Gbar)\cong V\rtimes R\cong A_4$.  This proves Step~B2.

\smallskip%\noindent
We now define some more notation. 
\begin{itemize}[leftmargin=\parindent]
\item[(1)] %It follows from
Step~B2 implies that $|G:F(G)|=3$. Hence the unique Sylow $2$-subgroup $S$ of $F(G)$ is the unique Sylow $2$-subgroup of $G$, and $\overline{S}=P$. 
Further, by Step~B2,  $G$ has four Sylow 
$3$-subgroups; let $T$ be one of them, and suppose without loss of generality that $\overline{T}=Q\times R$. Then $G=S\rtimes T$, and $T_0:=T\cap F(G)$ is the unique
Sylow $3$-subgroup of $F(G)$ with $|T:T_0|=3$.  Thus $F(G)=S\times T_0$, $T_0=C_T(S)$, and $\overline{T_0}=Q$.
\item[(2)] By Step B2, $\Z(\Gbar)=U\times Q$ where $U=\Z(\Gbar)\cap P$, so $\overline{S}=P=U\times V$, and $C_F(R)=U\times Q$ (recall $F=\overline{F(G)} =P\times Q$).
\item[(3)] Let $\pi$ be the composition of the natural projection $\pi_0:G\to \Gbar$ and the projection $\pi_1:\Gbar\to \Gbar/(U\times Q)$.  Let $M:=\Ker(\pi)$ and 
$B:=\Ker(\pi)\cap S$. Since $S\lhdeq G$, both $M$ and $B$ are normal in $G$. Also $G/M\cong \pi(G)= \Gbar/(U\times Q) \cong V\rtimes R\cong A_4$, and $S/B\cong\pi(S)=\pi_1(P)\cong V\cong (C_2)^2$, and $\pi(T)\cong \pi_1(T)/Q\cong R\cong C_3$. In particular $T$ permutes cyclically the three non-trivial elements of $S/B$. 
\item[(4)] In fact $T\cap \Ker(\pi)=T_0$ so $M=B\times T_0\leq F(G)$. We let $H:=BT$ so 
$|G:H|=|ST:BT|=|S:B|=4$ and $|H:M|=3$.
\end{itemize}

\smallskip%\noindent
{\bf Step B3.} \emph{If $x\in H\setminus M$, then $C_G(x)=H$ and $|x^G|=4$.}

%{\bf Note: this is different from [BGP] where they can only conclude that $C_G(x)\leq H$ with index at most 2, and applying [BGP, Lemma 1.1], $|H'|\leq 2$.}

Let $x\in H\setminus M$, so $x=bt$ for unique $b\in B, t\in T\setminus T_0$. Suppose that $y\in C_G(x)$. Thus $[y,x]=1$ and $y=sr$ for unique
$s\in S, r\in T$. Since $[sr,bt]=1$, computing modulo the normal subgroup $B$
shows that $[sr,t]\in B$. However,
\[
[sr,t] = r^{-1}s^{-1}t^{-1}srt = r^{-1}[s,t] t^{-1}rt = [s,t]^r[r,t].
\]
Further, $[s,t]^r\in S$ (since $S\lhdeq G$) and $[r,t]\in T$. Since $[sr,t]\in B\subseteq S$ it follows that $[r,t]=1$, and hence that $[s,t]^r\in B$ which implies that $[s,t]\in B$ (since $B\lhdeq G$). 

We claim that $s\in B$. If not, then $s\in S\setminus B$, and since $t\in T\setminus T_0$, the element $t$ maps the non-trivial coset $sB$ to $(sB)^t = s^tB\ne sB$, implying that $s^{-1}s^t = [s,t]\not\in B$, a contradiction. Thus $s\in B$, and hence $y=sr\in BT=H$.  This means that $|x^G|=|G:C_G(x)|$ is divisible by $|G:H|=4$, and since $\cs(G)=\{2,4,6\}$ we conclude that $|x^G|=4$ and $H=C_G(x)$.

\smallskip%\noindent
{\bf Step B4.} \emph{With the above notation, $H=B\times T$ is abelian, and $B=C_S(T)$. Further $S=[S,T]\circ B$ is a central product, and $B\leq \Z(S)$.}

By Step B3, $H$ centralises each element of $H\setminus M$. Let $h\in M$ and $x\in
H\setminus M$. Then $xh\in H\setminus M$ so $H$ centralises both
$x^{-1}$ and $xh$, and hence $H$ also centralises
$x^{-1}(xh)=h$. Thus $H$ is abelian, and hence $H=B\times
T$, with $B, T$ abelian. Next, since $T$ acts fixed point freely on
$S/B$ and centralises $B$, it follows that $B=C_S(T)$.

Now $S=[S,T]C_S(T)$ by~\cite{I}*{Lemma~4.28}, and since $B=C_S(T)$ we have
$S=[S,T]B$. Using the Three-Subgroup Lemma~\cite{I}*{Lemma~4.9}, since $[[T,B],S]=[1,S]=1$ and $[[B,S],T]\leq [B,T]=1$, we conclude that $[[S,T],B]=1$. Thus $S=[S,T]\circ B$ is a central product, and as $B$ is abelian, this implies $B\leq \Z(S)$.

\smallskip%\noindent
{\bf Step B5.} \emph{Set $Z:=[S,T]\cap B$. Then $[S,T]/Z\cong S/B\cong  (C_2)^2$,
$S'=[S,T]'\cong C_2$, $\Z(S)=B$, $\Z([S,T])=Z$, and
$\Z(G)=B\times T_0=M$.}

By definition, since $Z\leq B$ it follows from Step B4 that $Z$ centralises $[S,T]$, and since also $B$ is abelian, we have $Z\leq \Z(S)\leq \Z([S,T])$. Also, by Step B4 and (3) above,
\[
[S,T]/Z =[S,T]/([S,T]\cap B)\cong ([S,T]B)/B = S/B \cong C_2^2.
\]
Let $a\in [S,T]\setminus Z$ (or more generally let $a\in S\setminus
B$) and $x\in T\setminus T_0$. Then $x$ acts fixed point freely on
$S/B\cong [S,T]/Z$ so, under the homomorphism $\pi:G\rightarrow A_4$
defined in (3) above, $\pi(a)$ is an involution in the fours group $\pi(S)$  of
$A_4$, and $\pi(x)$ acts fixed-point freely on $\pi(S)$. Hence
$\pi(a)$ lies in an $A_4$-conjugacy class of size $3$, and so $|a^G|$
is divisible by $3$ and hence is equal to $6$, since
$\cs(G)=\{2,4,6\}$. It follows that $C_G(a)\leq F(G) =S\times T_0$,
and hence $C_G(a)=C_{F(G)}(a)=C_S(a)\times T_0$ with $C_S(a)$ of
index $2$ in $S$. By Step B4, $B\leq \Z(S)$, and so $C_S(a) = \la B, a\ra$,
of index $2$ in $S$. In
particular $C_S(a)$ does not contain $[S,T]$. Thus $[S,T]$ is
nonabelian, and for each $a\in [S,T]\setminus Z$, $C_{[S,T]}(a)$ has
index $2$ in $[S,T]$; also if $a\in S\setminus B$ then $C_S(a)$ has
index $2$ in $S$. Applying~\cite{BGP}*{Lemma 1.1} to the nonabelian group
$[S,T]$ with proper subgroup $Z$, and to the nonabelian group $S$ with
proper subgroup $B$, we conclude that $|[S,T]'|=|S'|=2$. Thus
$S'=[S,T]'\cong C_2$.

We have just shown that no element of $[S,T]\setminus Z$ is central in
$[S,T]$, and hence $\Z([S,T])\subseteq Z$. We proved the reverse inclusion 
above, and hence $\Z([S,T])=Z$.  An analogous argument shows that $\Z(S)=B$. Finally,
$B$ centralises both $S$ and $T$ by Step B4, so $B\leq \Z(G)$.
Also $T_0=C_T(S)$ centralises both $S$ and $T$ (since $T$ is abelian by Step B4), so $T_0\leq\Z(G)$.
Hence $M=B\times T_0\leq \Z(G)$ and equality holds since
$G/M\cong A_4$ has trivial centre.

\smallskip%\noindent
{\bf Step B6.}
We can now eliminate Case~B. Since $\cs^*(G/\Z(G))=\cs^*(A_4)=\{3,4\}$, it follows that any conjugacy class of $G$ of size~2
is contained in $\Z(G)$. This is a contradiction as every conjugacy
class in $\Z(G)$ has size~1. This eliminates Case~B and completes
our rather long proof.
\end{proof}

%\section*{Acknowledgement}


\begin{thebibliography}{99}

%\bib{A}{book}{
%   author={Aschbacher, Michael},
%   title={Finite group theory},
%   series={Cambridge Studies in Advanced Mathematics},
%   volume={10},
%   note={Corrected reprint of the 1986 original},
%   publisher={Cambridge University Press, Cambridge},
%   date={1993},
%   pages={x+274},
%   isbn={0-521-45826-9},
%   review={\MR{1264416}},
%}


\bib{BP}{article}{
   author={Balog, Antal},
   author={Pomerance, Carl},
   title={The distribution of smooth numbers in arithmetic progressions},
   journal={Proc. Amer. Math. Soc.},
   volume={115},
   date={1992},
   number={1},
   pages={33--43},
   issn={0002-9939},
%   review={\MR{1089401}},
%   doi={10.2307/2159561},
}

\bib{B92}{article}{%MR1143160
   author={Bianchi, Mariagrazia},
   author={Chillag, David},
   author={Mauri, Anna Gillio Berta},
   author={Herzog, Marcel},
   author={Scoppola, Carlo M.},
   title={Applications of a graph related to conjugacy classes in finite
   groups},
   journal={Arch. Math. (Basel)},
   volume={58},
   date={1992},
   number={2},
   pages={126--132},
   issn={0003-889X},
%   review={\MR{1143160}},
%   doi={10.1007/BF01191876},
}
	

\bib{BGC}{article}{
   author={Bianchi, Mariagrazia},
   author={Gillio, Anna},
   author={Casolo, Carlo},
   title={A note on conjugacy class sizes of finite groups},
   journal={Rend. Sem. Mat. Univ. Padova},
   volume={106},
   date={2001},
   pages={255--260},
   issn={0041-8994},
%   review={\MR{1876222}},
}

\bib{BGP}{article}{
   author={Bianchi, M.},
   author={Gillio, A.},
   author={P\'{a}lfy, P. P.},
   title={A note on finite groups in which the conjugacy class sizes form an
   arithmetic progression},
   conference={
      title={Ischia group theory 2010},
   },
   book={
      publisher={World Sci. Publ., Hackensack, NJ},
   },
   date={2012},
   pages={20--25},
%   review={\MR{3184980}},
%   doi={10.1142/9789814350051_0003},
}
	
\bib{Bosma}{article}{
   author={Bosma, Wieb},
   author={Cannon, John},
   author={Playoust, Catherine},
   title={The Magma algebra system. I. The user language},
   note={Computational algebra and number theory (London, 1993)},
   journal={J. Symbolic Comput.},
   volume={24},
   date={1997},
   number={3-4},
   pages={235--265},
   issn={0747-7171},
%   review={\MR{1484478}},
%   doi={10.1006/jsco.1996.0125},
}
	

%\bib{CC}{article}{
%   author={Camina, A. R.},
%   author={Camina, R. D.},
%   title={Coprime conjugacy class sizes},
%   journal={Asian-Eur. J. Math.},
%   volume={2},
%   date={2009},
%   number={2},
%   pages={183--190},
%   issn={1793-5571},
%   review={\MR{2532697}},
%   doi={10.1142/S1793557109000157},
%}

\bib{Camina}{article}{
   author={Camina, A. R.},
   author={Camina, R. D.},
   title={The influence of conjugacy class sizes on the structure of finite
   groups: a survey},
   journal={Asian-Eur. J. Math.},
   volume={4},
   date={2011},
   number={4},
   pages={559--588},
   issn={1793-5571},
%   review={\MR{2875589}},
%   doi={10.1142/S1793557111000459},
}

\bib{CH}{article}{
   author={Chillag, David},
   author={Herzog, Marcel},
   title={On the length of the conjugacy classes of finite groups},
   journal={J. Algebra},
   volume={131},
   date={1990},
   number={1},
   pages={110--125},
   issn={0021-8693},
%   review={\MR{1055001}},
%   doi={10.1016/0021-8693(90)90168-N},
}

\bib{CosseyHawkes}{article}{
   author={Cossey, John},
   author={Hawkes, Trevor},
   title={Sets of $p$-powers as conjugacy class sizes},
   journal={Proc. Amer. Math. Soc.},
   volume={128},
   date={2000},
   number={1},
   pages={49--51},
   issn={0002-9939},
%   review={\MR{1641677}},
%   doi={10.1090/S0002-9939-99-05138-2},
}


\bib{DJ}{article}{
   author={Dolfi, Silvio},
   author={Jabara, Enrico},
   title={The structure of finite groups of conjugate rank $2$},
   journal={Bull. Lond. Math. Soc.},
   volume={41},
   date={2009},
   number={5},
   pages={916--926},
   issn={0024-6093},
%   review={\MR{2557472}},
%   doi={10.1112/blms/bdp072},
}

%\bib{G}{book}{
%   author={Gorenstein, Daniel},
%   title={Finite groups},
%   edition={2},
%   publisher={Chelsea Publishing Co., New York},
%   date={1980},
%   pages={xvii+519},
%   isbn={0-8284-0301-5},
%   review={\MR{569209}},
%}

\bib{GT}{article}{
   author={Green, Ben},
   author={Tao, Terence},
   title={The primes contain arbitrarily long arithmetic progressions},
   journal={Ann. of Math. (2)},
   volume={167},
   date={2008},
   number={2},
   pages={481--547},
   issn={0003-486X},
%   review={\MR{2415379}},
%   doi={10.4007/annals.2008.167.481},
}

\bib{Huppert}{book}{
   author={Huppert, B.},
   title={Endliche Gruppen. I},
   language={German},
   series={Die Grundlehren der Mathematischen Wissenschaften, Band 134},
   publisher={Springer-Verlag, Berlin-New York},
   date={1967},
   pages={xii+793},
%   review={\MR{0224703}},
}


\bib{H}{article}{
   author={Huppert, Bertram},
   title={Character degrees},
   note={Group theory and combinatorial geometry},
   journal={Rend. Circ. Mat. Palermo (2) Suppl.},
   number={19},
   date={1988},
   pages={113--118},
   issn={1592-9531},
%   review={\MR{988192}},
}

\bib{bH}{article}{
   author={Huppert, Bertram},
   title={A characterization of ${\rm GL}(2,3)$ and ${\rm SL}(2,5)$ by the
   degrees of their representations},
   journal={Forum Math.},
   volume={1},
   date={1989},
   number={2},
   pages={167--183},
   issn={0933-7741},
%   review={\MR{990142}},
%   doi={10.1515/form.1989.1.167},
}

\bib{I}{book}{
   author={Isaacs, I. Martin},
   title={Finite group theory},
   series={Graduate Studies in Mathematics},
   volume={92},
   publisher={American Mathematical Society, Providence, RI},
   date={2008},
   pages={xii+350},
   isbn={978-0-8218-4344-4},
%   review={\MR{2426855}},
%   doi={10.1090/gsm/092},
}
	
\bib{Ito1}{article}{
   author={It\^{o}, Noboru},
   title={On finite groups with given conjugate types. I},
   journal={Nagoya Math. J.},
   volume={6},
   date={1953},
   pages={17--28},
   issn={0027-7630},
%   review={\MR{0061597}},
}

\bib{Ito2}{article}{
   author={It\^{o}, Noboru},
   title={On finite groups with given conjugate types. II},
   journal={Osaka J. Math.},
   volume={7},
   date={1970},
   pages={231--251},
   issn={0030-6126},
%   review={\MR{0263910}},
}

\bib{Ito3}{article}{
   author={It\^{o}, Noboru},
   title={On finite groups with given conjugate types. III},
   journal={Math. Z.},
   volume={117},
   date={1970},
   pages={267--271},
   issn={0025-5874},
%   review={\MR{0289625}},
%   doi={10.1007/BF01109847},
}

\bib{K}{article}{
   author={Kazarin, L. S.},
   title={Burnside's $p^\alpha$-lemma},
   language={Russian},
   journal={Mat. Zametki},
   volume={48},
   date={1990},
   number={2},
   pages={45--48, 158},
   issn={0025-567X},
   translation={
      journal={Math. Notes},
      volume={48},
      date={1990},
      number={1-2},
      pages={749--751 (1991)},
      issn={0001-4346},
   },
%   review={\MR{1076932}},
%   doi={10.1007/BF01262606},
}

\bib{L}{article}{
   author={Lewis, Mark L.},
   title={An overview of graphs associated with character degrees and
   conjugacy class sizes in finite groups},
   journal={Rocky Mountain J. Math.},
   volume={38},
   date={2008},
   number={1},
   pages={175--211},
   issn={0035-7596},
%   review={\MR{2397031}},
%   doi={10.1216/RMJ-2008-38-1-175},
}


\bib{T}{article}{
   author={Thompson, John G.},
   title={A special class of non-solvable groups},
   journal={Math. Z},
   volume={72},
   date={1959/1960},
   pages={458--462},
   issn={0025-5874},
%   review={\MR{0117288}},
%   doi={10.1007/BF01162968},
}

\end{thebibliography}
\end{document}